\documentclass[final,1p,times]{elsarticle}

\usepackage{amsmath,amsthm,amssymb,amscd}
\usepackage{mathrsfs}

\journal{Journal of Functional Analysis}

\newtheorem{theorem}{Theorem}[section]
\newtheorem{lemma}[theorem]{Lemma}
\newtheorem{proposition}[theorem]{Proposition}

\newtheorem{remark}[theorem]{Remark}
\newtheorem{definition}[theorem]{Definition}

\newcommand{\cJ}{{\mathcal J}}
\DeclareMathOperator *{\diam}{diam}
\DeclareMathOperator *{\dist}{dist}
\DeclareMathOperator *{\Tr}{Tr}
\DeclareMathOperator *{\BBR}{\mathbb{R}}
\DeclareMathOperator *{\BBC}{\mathbb{C}}

\def\Xint#1{\mathchoice
{\XXint\displaystyle\textstyle{#1}}%
{\XXint\textstyle\scriptstyle{#1}}%
{\XXint\scriptstyle\scriptscriptstyle{#1}}%
{\XXint\scriptscriptstyle\scriptscriptstyle{#1}}%
\!\int}
\def\XXint#1#2#3{{\setbox0=\hbox{$#1{#2#3}{\int}$ }
\vcenter{\hbox{$#2#3$ }}\kern-.6\wd0}}

\def\dint{\Xint-}

\begin{document}

\begin{frontmatter}

\title{
Extrapolation of the Dirichlet problem for elliptic equations with complex coefficients}

\author{Martin Dindo\v{s}}
\address{School of Mathematics,
         The University of Edinburgh and Maxwell Institute of Mathematical Sciences, UK}
\ead{M.Dindos@ed.ac.uk}

\author{Jill Pipher}
\address{Department of Mathematics, 
	Brown University, USA}
\ead{jill\_pipher@brown.edu}

\begin{abstract}
In this paper, we prove an extrapolation result for complex coefficient 
divergence form operators that satisfy a strong ellipticity condition known as $p$-{\it ellipticity}. 
Specifically, let $\Omega$ be a chord-arc domain in $\mathbb R^n$ and the operator $\mathcal L = \partial_{i}\left(A_{ij}(x)\partial_{j}\right) 
+B_{i}(x)\partial_{i} $ be elliptic, with $|B_i(x)| \le K\delta(x)^{-1}$ for a small $K$. Let
$p_0 = \sup\{p>1: A \,\,\mbox{is}\,\, \mbox{$p$-elliptic}\}$.

We establish that if the $L^q$ Dirichlet problem is solvable for $\mathcal L$ for some $1<q< \frac{p_0(n-1)}{(n-2)}$, then the $L^p$ Dirichlet problem is solvable for all $p$ in the range $\left[q,  \frac{p_0(n-1)}{(n-2)}\right)$. 
In particular, if the matrix $A$ is real, or $n=2$, the $L^p$ Dirichlet problem is solvable for $p$ in the range $[q, \infty)$.
\end{abstract}


\end{frontmatter}

\section{Introduction}\label{S:Intro}
Over the past several decades, a well developed theory of solvability of boundary value problems for real second order elliptic and parabolic
equations has evolved, a theory that connects and quantifies the range of solvability with a variety of ways of measuring smoothness of 
the coefficients and of the boundary domain. While the literature is vast, some early advances in this area include \cite{CFMS},
\cite{Dah}, \cite{DaJ}, \cite{FKP}, \cite{FJR}, \cite{JK}, \cite{JKRellich} \cite{LHannals}, and \cite{Ve};  for a small sample of some more
recent contributions,
we point to \cite{AHLMT}, \cite{AHMNT}, \cite{DaFM}, \cite{DHw}, \cite{DHM}, \cite{DKP}, \cite{DPR}, \cite{HKMP}, \cite{HMM}, \cite{HMT}, \cite{KP}, \cite{KKPT},  \cite{LHmem}, 
and \cite{To}. Equally important are boundary value problems for systems of equations, higher order equations, and second order equations with complex coefficients, but solvability for these equations presents many more challenges. The main challenges to a comparably complete
understanding in these three settings are the lack of regularity of solutions, such as that guaranteed in the real valued setting
 by the De Giorgi-Nash-Moser theory, and the lack of even a weak
(Agmon-Miranda) maximum principle. Much of our understanding of solvability of real second order elliptic/parabolic equations, and
how solvability for particular function spaces of boundary data connects to
the geometry of the domain, depends on these principles. 

In this paper, we take up the question of extrapolation of the solvability of a particular boundary value, the Dirichlet problem, for complex
coefficient divergence form elliptic operators. We use the term {\it extrapolation} to mean that solvability of the Dirichlet
problem for boundary data in one function space implies solvability in a range of function spaces. Extrapolation is not 
possible for arbitrary elliptic complex coefficient operators. Our goal in this paper is to provide natural and checkable structural conditions on
the operator for which extrapolation holds. Many ideas we develop here are related to recent work of Shen \cite{S2} on extrapolation of systems of elliptic PDEs in Lipschitz domains, which we became aware of after this paper was submitted. In forthcoming work, we have introduced the notion of $p$-ellipticity 
for more general second order complex coefficient elliptic systems, permitting an extension of Shen's results for these systems.

The main result of this paper is an extrapolation result for complex coefficient 
divergence form operators that satisfy a strong ellipticity condition known as $p$-{\it ellipticity}. 
Essentially, $p$-ellipticity measures how close the operator is to being real-valued. The interval where an operator is $p$-elliptic depends on the size of the imaginary part of the coefficients. If the interval of $p$-elipticity is $(1,\infty)$, then the operator is real-valued.
We will discuss this
condition in more detail below.

Our boundary value problems are formulated for measurable data in a Lebesgue space; solvability is
described in terms of nontangential convergence and a priori estimates on a nontangential maximal function.
The complex-valued setting is very different from the real-valued theory that has been well developed 
over several decades since the fundamental regularity results of De Giorgi - Nash - Moser.
It is well known that real-valued second order elliptic operators in divergence form satisfy a maximum principle;
in the language of nontangential boundary value problems, this principle translates into solvability of the Dirichlet
problem with data in $L^\infty$ on the boundary of the domain. 
This in turn entails that solvability of boundary value problems with data in a Lebesgue space $L^q(\partial \Omega)$ {\it extrapolates},
via interpolation with the endpoint $L^\infty$, to solvability in all $L^p$, $p\geq q$.
There are many techniques that can establish solvability for a single value of $p$, for example the Kato-type techniques
for a special class of complex operators (\cite{AAAHK}), the Rellich-type inequalities for real symmetric operators
(\cite{JKRellich}), or methods such as \cite{KP01} and \cite{DPP} 
for coefficients satisfying a Carleson measure condition. But in the complex valued case, 
there is no maximum principle in general. For this reason, aside from this paper and \cite{S2}, 
extrapolation results have only been shown in the presence of
$L^\infty$ estimates for solutions, for example, in the limited setting of small perturbations of real-valued operators.

We now give some background for the results in this paper.
The work in \cite{DPcplx} initiated the study of higher regularity of solutions to {\it complex} operators of the form
\begin{equation}
\mathcal L= \partial_{i}\left(A_{ij}(x)\partial_{j}\right) 
+B_{i}(x)\partial_{i}
\end{equation}
 where $A:= (A_{ij})$ is uniformly elliptic and bounded and $|B_i(x)| \le K\delta(x)^{-1}$, under a
structural assumption called $p$-ellipticity.  These new regularity results were used to establish solvability of the Dirichlet
problem for a certain class of such operators in domains $\Omega$ with boundary data in $L^q(\partial \Omega)$ for $q$ in the
range of $p$-ellipticity. In \cite{FMZ}, the authors gave another proof of the interior higher regularity results of \cite{DPcplx}.
They were then able to prove boundary regularity estimates
for domains $\Omega$ satisfying certain minimal geometric conditions. 
In this paper, we use the interior regularity and its extension to the boundary to prove 
the main theorem.

\begin{theorem}\label{MainThm}
Let $\Omega$ be a chord-arc domain in $\mathbb R^n$ and $\mathcal L = \partial_{i}\left(A_{ij}(x)\partial_{j}\right) 
+B_{i}(x)\partial_{i} $ be a second order operator with bounded and measurable coefficients $A$ and $|B|\lesssim \delta(x)^{-1}$. Define 
$$p_0 = \sup\{p>1: A \,\,\mbox{is}\,\, \mbox{p-elliptic}\}.$$ 
Assume that the $L^q$ Dirichlet problem is solvable for $\mathcal L$ for some $q\in (1,\frac{p_0(n-1)}{(n-2)})$  (if $p_0 = \infty$ or $n=2$ we require $q\in (1,\infty)$).

\noindent Then the $L^p$ Dirichlet problem is solvable for 
$\mathcal L$ for $p$ in the range $[q,  \frac{p_0(n-1)}{(n-2)})$, if one of the following constraints holds on the size of the vector $B$.
\begin{itemize}
\item $\Omega$ is bounded and $B(x)=o(\delta(x)^{-1})$ as $x\to\partial\Omega$.
\item $\Omega$ is bounded and $\limsup_{x\to\partial\Omega} |B(x)\delta(x)|\le K$. Here $K=K(A,p,n)>0$ is sufficiently small. 
\item $\Omega$ is unbounded and $|B(x)\delta(x)|\le K$ for all $x\in\Omega$. Here $K=K(A,p,n)>0$ is sufficiently small. 
\end{itemize}

\end{theorem}

In particular, when $p_0 = \infty$, i.e., the matrix $A$ is real, the $L^p$ Dirichlet problem is solvable for $p$ in the range $[q, \infty)$. The same 
conclusion holds when the dimension $n=2$, since the main inequalities of \eqref{eq-pf13} hold for all $r$.

In the next section, we discuss the background in more detail and define the terms used in the statement of the theorem. In section 3 we give the proof of the main theorem.

\section{Background and definitions}
\subsection{$p$-ellipticity}
A concept related to $p$-ellipticity
was introduced in \cite{CM}, where the authors investigated the $L^p$-dissipativity of second order divergence complex coefficient operators. Later, and
independently, we (\cite{DPcplx}) and Carbonaro and Dragi\v{c}evi\'c (\cite{CD}) gave equivalent definitions
of this property - the term ``$p$-ellipticity" was coined in \cite{CD} and their definition is the one we introduce below.
To introduce this, we define, for $p>1$, the ${\mathbb R}$-linear map $\cJ_p:{\mathbb C}^n\to {\mathbb C}^n$ by
$$\cJ_p(\alpha+i\beta)=\frac{\alpha}{p}+i\frac{\beta}{p'}$$
where $p'=p/(p-1)$ and $\alpha,\beta\in{\mathbb R}^n$.

\begin{definition}\label{pellipticity} Let $\Omega\subset{\mathbb R}^n$. Let $A:\Omega\to M_n(\mathbb C)$, where $M_n(\mathbb C)$ is the space of $n\times n$ complex valued matrices. We say that $A$ is $p$-elliptic if for a.e. $x\in\Omega$
\begin{equation}\label{pEll}
\mathscr{R}e\,\langle A(x)\xi,\cJ_p\xi\rangle \ge \lambda_p|\xi|^2,\qquad\forall \xi\in{\mathbb C}^n
\end{equation}
for some $\lambda_p>0$ and there exists $\Lambda>0$ such that 
\begin{equation}
|\langle A(x)\xi,\eta \rangle| \le \Lambda |\xi| |\eta|, \qquad\forall \xi, \,\eta\in{\mathbb C}^n.
\end{equation}
\end{definition}

It is now easy to observe that the notion of $2$-ellipticity coincides with the usual ellipticity condition for complex matrices.
As shown in \cite{CD} if $A$ is elliptic, then there exists $\mu(A)>0$ such that $A$ is $p$-elliptic if and only if
$\left|1-\frac2p\right|<\mu(A).$
Also $\mu(A)=1$ if and only if $A$ is real valued.\medskip

We give some notation. 
Here and in what follows we will use the convention that points in the interior of $\Omega$ will be denoted by uncapitalised letters such as $x,y$; while points on the boundary will be denoted by the capital letters such as $P$ or $Q$. The
expression $\Delta(Q,r):= B(Q,r) \cap \partial \Omega$ will be used to denote the surface ball centered at $Q$ of radius $r$ contained in the boundary of $\Omega$. The Carleson region associated to $\Delta(Q,r)$ is defined to be $T(\Delta):= B(Q,r) \cap \Omega$.

\subsection{Chord-arc domains (CAD)}

Our aim is to establish the extrapolation result under minimal necessary assumptions on the geometry of the domain $\Omega\subset{\mathbb R}^n$ and its boundary $\partial\Omega$. Recently, there has been substantial progress in understanding the interplay between boundary regularity of the domain and solvability of boundary value problems for elliptic operators. We start by collecting some definitions.

\begin{definition}[\bf Corkscrew condition]\label{def1.cork}
 \cite{JK}.  A domain $\Omega\subset {\mathbb R}^{n}$
satisfies the {\it Corkscrew condition} if for some uniform constant $c>0$ and
for every surface ball $\Delta:=\Delta(Q,r),$ with $Q\in \partial\Omega$ and
$0<r<\diam(\partial\Omega)$, there is a ball
$B(x_\Delta,cr)\subset B(Q,r)\cap\Omega$.  The point $x_\Delta\subset \Omega$ is called
a {\it corkscrew point relative to} $\Delta,$ (or, relative to $B$). We note that  we may allow
$r<C\diam(\partial\Omega)$ for any fixed $C$, simply by adjusting the constant $c$.
\end{definition}

\begin{definition}[\bf Harnack Chain condition]\label{def1.hc}
\cite{JK}. Let $\delta(x)$ denote the distance of $x \in \Omega$ to $\partial \Omega$. A domain
$\Omega$ satisfies the {\it Harnack Chain condition} if there is a uniform constant $C$ such that
for every $\rho >0,\, \Lambda\geq 1$, and every pair of points
$x,x' \in \Omega$ with $\delta(x),\,\delta(x') \geq\rho$ and $|x-x'|<\Lambda\,\rho$, there is a chain of
open balls
$B_1,\dots,B_N \subset \Omega$, $N\leq C(\Lambda)$,
with $x\in B_1,\, x'\in B_N,$ $B_k\cap B_{k+1}\neq \emptyset$
and $C^{-1}\diam (B_k) \leq \dist (B_k,\partial\Omega)\leq C\diam (B_k).$  The chain of balls is called
a {\it Harnack Chain}.
\end{definition}

\begin{definition}[\bf 1-sided NTA]\label{def1.1nta}
If $\Omega$ satisfies both the Corkscrew and Harnack Chain conditions, then
$\Omega$ is a {\it 1-sided NTA domain} ($\Omega$ is sometimes called a {\it uniform domain}).
\end{definition}

\begin{definition}[\bf Ahlfors-David regular]\label{def1.ADR}
A closed set $E \subset {\mathbb R}^{n}$ is $n-1$-dimensional ADR (or simply ADR) (Ahlfors-David regular) if
there is some uniform constant $C$ such that for $\sigma=H^{n-1}$ (the $n-1$ dimensional Hausdorff measure)
\begin{equation} \label{eq1.ADR}
\frac1C\, r^{n-1} \leq \sigma(E\cap B(Q,r)) \leq C\, r^{n-1},\,\,\,\forall r\in(0,R_0),Q \in E,
\end{equation}
where $R_0$ is the diameter
of $E$ (which may be infinite).  
\end{definition}

Given that we are interested in solvability of boundary value problems on $n-1$ dimensional boundaries, it is natural to assume that our domain $\Omega$ is a 1-sided NTA domain with $n-1$-dimensional ADR boundary. However, it was established in \cite[Theorem 1.2]{AHMNT} that, even in the case of the simplest second order elliptic operator (the Laplacian), an extra assumption on the regularity of the domain is required. That is, the following result was proven:

\begin{theorem}\label{t2}
Suppose that $\Omega\subset \mathbb R^n$ is a 1-sided NTA (aka uniform) domain, whose boundary is Ahlfors-David regular.
Then the following are equivalent:
\begin{enumerate}\itemsep=0.05cm
\item $\partial\Omega$ is uniformly rectifiable.
\item $\Omega$ is an NTA domain (i.e., it is a 1-sided NTA which also satisfies the Corkscrew condition in the exterior $\mathbb R^n\setminus \overline{\Omega})$.
\item $\omega\in A_\infty$.
\end{enumerate}
\end{theorem}
Here $\omega$ denotes harmonic measure for $\partial\Omega$ with some fixed pole inside the domain. For the Laplacian, it is 
a classical fact that $\omega\in A_\infty$ is equivalent to solvability of the $L^p$ Dirichlet problem for some $p\in (1,\infty)$. 
Therefore, we shall assume that our domain $\Omega$ also satisfies the exterior Corkscrew condition. See also \cite{HMTo} for the variable coefficient version of this result.

\begin{definition}[\bf Chord-Arc domain]\label{def1cad}
If $\Omega$ and  $\mathbb R^n\setminus \overline{\Omega}$ satisfy the Corkscrew condition, 
$\Omega$ satisfies the Harnack Chain condition, and $\partial\Omega$ is is $n-1$-dimensional Ahlfors-David regular, then 
$\Omega$ is a chord-arc domain (CAD).
\end{definition}

We also note that on chord-arc domains there is a well defined notion of trace.  Let
$$W := \left\{u \in L^1_{loc}(\Omega):\, \|u\|_W := \left(\int_\Omega |\nabla u |^2 dx\right)^\frac12 < +\infty\right\},$$
which is clearly contained in $W_{loc}^{1,2}(\Omega)$. Then, if $\Omega$ is  CAD, there exists a bounded operator  $\Tr$ from $W$ to $L^2_{loc}(\partial\Omega,\sigma)$ such that $\Tr u = u\big|_{\partial\Omega}$ if $u\in W \cap C^0(\overline \Omega)$. Also the image $\Tr(W)$ is dense in $C^0(\overline \Omega)$.

Our notion of solvability of the Dirichlet problem requires a few more definitions.
In the first place, we need to introduce a non-standard notion of a nontangential approach region - such regions are typically referred to as
``cones" when the domain is at least Lipschitz regular, and ``corkscrew" regions when the domain is chord-arc. In the following, the parameter
$a$ is positive and will be referred to as the ``aperture". 
A standard corkscrew region associated with a boundary point $Q$ is defined (\cite{JK})  to be 
$$\gamma_a(Q)=\{x \in \Omega: |x-Q|< (1+a)\delta(x)\}$$
for some $a>0$ and nontangential maximal functions, square functions are defined in the literature with respect to these regions.
We modify this definition in order to achieve a certain geometric property (see Proposition \ref{Prop}) which may not hold 
for the $\gamma_a(Q)$ in general.

\begin{definition}\label{Cones}
For $y \in \Omega$, let $S_a(y) := \{Q \in \partial \Omega: y \in \gamma_a(Q)\}.$
Set 
$$\tilde{S}_a(y) := \bigcup_{Q \in S_a(y)} \Delta(Q, a\delta(y)).$$
Define 
$$\Gamma_a(Q) := \{ y \in \Omega: Q \in \tilde{S}_a(y)\}.$$
\end{definition}

Let us make some observations about this
novel definition of the corkscrew regions that we will use to 
define nontangential maximal functions.
We first note that, for any $Q \in \partial \Omega$,  $\gamma_a(Q) \subset \Gamma_a(Q)$. If $y \in \gamma_a(Q)$, then
$Q \in S_a(y) \subset \tilde{S}_a(y)$, i.e., $y \in \Gamma_a(Q)$. 
Next, we note that, for any $Q \in \partial \Omega$,  $\Gamma_a(Q) \subset \gamma_{2a}(Q)$. Indeed,
if $y \in \Gamma_a(Q)$, then $Q \in \tilde{S}_a(y)$ and therefore there exists a $Q_0 \in S_a(y)$ such that
$|Q-Q_0| < a\delta(y)$. Hence, 
$$|y-Q| \leq |y-Q_0| + |Q-Q_0| < (1+a) \delta(y) + a\delta(y) = (1+2a)\delta(y).$$
Thus our $\Gamma_a(Q)$ is sandwiched in between two standard corkscrew regions and is thus itself a 
corkscrew region.

\begin{definition}\label{D:NT} 
For $\Omega\subset\mathbb{R}^{n}$ as above, 
the nontangential maximal function $\tilde{N}_{p,a}$ is defined using $L^p$ averages over balls in the domain $\Omega$. 
Specifically, given $w\in L^p_{loc}(\Omega;{\BBC})$ we set
\begin{equation}\label{SSS-3}
\tilde{N}_{p,a}(w)(Q):=\sup_{x\in\Gamma_{a}(Q)}w_p(x)\,
\end{equation}
where, at each $x\in\Omega$, 
\begin{equation}\label{w}
w_p(x):=\left(\dint_{B_{\delta(x)/2}(x)}|w(z)|^{p}\,dz\right)^{1/p}.
\end{equation}
\end{definition}

The regions $\Gamma_a(Q)$ have the following property inherited from $\gamma_{2a}(Q)$: for any pair of points $x, x'$ in
$\Gamma_a(Q)$, there is a Harnack chain of balls connecting $x$ and $x'$ - see Definition \ref{def1.hc}. The centers of the balls
in this Harnack chain will be contained in a corkscrew region $\Gamma_{a'}(Q)$ of slightly larger aperture, where $a'$ depends only
the geometric constants in the definition of the domain.

\subsection{The $L^p$-Dirichlet problem}

We recall the definition of $L^p$ solvability of the Dirichlet problem for an operator 
 $\mathcal L = \partial_{i}\left(A_{ij}(x)\partial_{j}u\right) 
+B_{i}(x)\partial_{i}u $ where $A:= (A_{ij})$ is uniformly elliptic and bounded, and $|B_i(x)| \le K \delta(x)^{-1}$, for some small $K<\infty$ to be determined later.

In anticipation to formulating the $L^p$-Dirichlet problem we first recall the notion of classical solvability, 
via the Lax-Milgram lemma. Given a CAD-domain $\Omega$, consider the bilinear form 
$\mathcal B:{\dot{W}}^{1,2}(\Omega;\mathbb C)\times {\dot{W}}^{1,2}_0(\Omega;\mathbb C) \to\mathbb C$ defined by
\begin{equation}\label{eq-BF}
\mathcal B[u,w]=\int_{\Omega}\left[A_{ij}(x)\partial_ju(x)\partial_iw(x)
+B_i(x)\partial_iu(x) w(x)\right]\,dx.
\end{equation}
Clearly, $\mathcal B$ is bounded under the assumptions $A$ has entries in $L^\infty(\Omega)$ and that $B$ satisfies
$|B(x)|\le K\delta^{-1} (x)$. Indeed, for the second term this allows us to use the Cauchy-Schwarz inequality 
followed by an application of Hardy-Sobolev inequality
\begin{equation}\label{HSob}
\int_{\Omega}\frac{|w(x)|^2}{\delta(x)^2}dx\le C\int_\Omega |\nabla w|^2\,dx
\end{equation}
which holds for each function $w\in \dot{W}^{1,2}_0(\Omega;\mathbb C)$. Recalling our earlier discussion for $W={\dot{W}}^{1,2}(\Omega;\mathbb C)$ we denoted by $\Tr(W)$ traces of functions from $W$ on $\Omega$ and noted that  $\Tr(W)$ is dense in $C^0(\overline{\Omega})$.

Given an arbitrary $f\in\Tr(W)$,  
there exists $v\in \dot{W}^{1,2}(\Omega;\mathbb C)$ such that ${\rm Tr}\,v=f$ on $\partial\Omega$. 
Writing $u=u_0+v$, we seek $u_0\in \dot{W}^{1,2}_0(\Omega;\mathbb C)$ such that
$$
\mathcal B[u_0,w]=-B[v,w]\,\,\,\mbox{ for all }\,\,w\in \dot{W}^{1,2}_0(\Omega;\mathbb C).
$$
Observe that $-B[v,\cdot]\in\big(\dot{W}^{1,2}_0(\Omega;\mathbb C)\big)^*$, hence by the Lax-Milgram lemma 
there exists unique solution $u_0\in \dot{W}^{1,2}_0(\Omega;\mathbb C)$, provided the form $\mathcal B$ 
is coercive on the space $\dot{W}^{1,2}_0(\Omega;\mathbb C)$.

When $\mathcal L$ is uniformly elliptic (i.e. $2$-elliptic) we clearly have
$$
Re\,\int_{\Omega}A_{ij}\partial_ju\overline{\partial_iu}\,dx\ge \lambda \int_\Omega|\nabla u|^2\,dx,
$$
for all $u\in \dot{W}^{1,2}_0(\Omega;\mathbb C)$. On the other hand, for the term involving the entries of 
$B$ we may use \eqref{HSob} to estimate 
$$
\left|\int_{\Omega}B_i(\partial_iu) \overline{u}\,dx \right|\le CK\int_\Omega|\nabla u|^2\,dx,
$$
hence
$$
{\mathcal B}[u,\overline{u}]\ge (\lambda-CK)\|\nabla u\|^2_{L^2(\Omega)}.
$$
This implies coercivity of the bilinear form $\mathcal B$, for small values of $K$. It follows that 
the Lax-Milgram lemma can be applied and guarantees the existence of weak solutions.
That is, 
given any $f\in \Tr(W)$, the homogenous space of traces of functions in $W$, there exists a unique 
$u \in W$ (up to a constant) such that $\mathcal{L}u=0$ in $\Omega$ and ${\rm Tr}\,u=f$ on $\partial\Omega$. We call these solutions {\it energy solutions} and use them to define the notion of solvability of the $L^p$ Dirichlet problem.

\begin{definition}\label{D:Dirichlet} 
Let $\Omega$ be a chord-arc domain in $\mathbb R^n$
and fix an integrability exponent 
$p\in(1,\infty)$. Also, fix an aperture parameter $a>0$. Consider the following Dirichlet problem 
for a complex valued function $u:\Omega\to{\BBC}$:
\begin{equation}\label{E:D}
\begin{cases}
0=\partial_{i}\left(A_{ij}(x)\partial_{j}u\right) 
+B_{i}(x)\partial_{i}u 
& \mbox{in } \Omega,
\\[4pt]
u(Q)=f(Q) & \mbox{ for $\sigma$-a.e. }\,Q\in\partial\Omega, 
\\[4pt]
\tilde{N}_{2,a}(u) \in L^{p}(\partial \Omega), &
\end{cases}
\end{equation}
where the usual Einstein summation convention over repeated indices ($i,j$ in this case) 
is employed. 

The Dirichlet problem \eqref{E:D} is solvable for a given $p\in(1,\infty)$ if  there exists a
$C=C(p,\Omega)>0$ such that
for all boundary data
$f\in L^p(\partial\Omega;{\BBC})\cap \Tr(W)$ the unique energy solution
satisfies the estimate
\begin{equation}\label{y7tGV}
\|\tilde{N}_{2,a} (u)\|_{L^{p}(\partial\Omega; d\sigma)}\leq C\|f\|_{L^{p}(\partial\Omega;d\sigma)},
\end{equation}
where $d\sigma$ denotes surface measure on the boundary, i.e., the restriction of $H^{n-1}$ to $\partial \Omega$.
\end{definition}

Above and elsewhere, a barred integral indicates an averaging operation. Observe that, given $w\in L^p_{loc}(\Omega;{\BBC})$, the function $w_p$ 
associated with $w$ as in \eqref{w} is continuous.
The $L^2$-averaged nontangential maximal function was introduced in \cite{KP2} in connection with
the Neuman and regularity problems. In the context of $p$-ellipticity, Proposition 3.5 of \cite{DPcplx} shows that there is no difference between 
$L^2$ averages and $L^p$ averages when $w=u$ solves $\mathcal Lu=0$ and that $\tilde{N}_{p,a}(u)$ and $\tilde{N}_{2,a'}(u)$ are comparable in $L^r$ norms for all $r>0$ and all allowable apertures $a,a'$.

\noindent{\it Remark.}  Given $f\in L^p(\partial\Omega;{\BBC})\cap\Tr(W)$,
the corresponding energy solution constructed above is unique: the decay implied by the $L^p$ estimates eliminates constant solutions. As the space 
$L^p(\partial\Omega;{\BBC})\cap\Tr(W)$ is dense in $C^0(\overline{W})$ and hence in
$L^p(\partial\Omega;{\BBC})$ for each $p\in(1,\infty)$, it follows that there exists a 
unique continuous extension of the solution operator
$f\mapsto u$
to the whole space $L^p(\partial\Omega;{\BBC})$, with $u$ such that $\tilde{N}_{2,a} (u)\in L^p(\partial\Omega)$ 
and, moreover, $\|\tilde{N}_{2,a} (u) \|_{L^{p}(\partial \Omega)} 
\leq C\|f\|_{L^{p}(\partial\Omega;{\BBC})}$. It was shown in the Appendix (section 7)  of \cite{DPcplx} that 
for any $f\in L^p(\partial \Omega;\mathbb C)$ the corresponding solution $u$ constructed by the continuous extension attains the datum $f$ as its boundary values in the following sense.
Consider the average $\tilde u:\Omega\to \mathbb C$ defined by
$$\tilde{u}(x)=\dint_{B_{\delta(x)/2}(x)} u(y)\,dy,\quad \forall x\in \Omega.$$
Then 
\begin{equation}
f(Q)=\lim_{x\to Q,\,x\in\Gamma(Q)}\tilde u(x),\qquad\mbox{for a.e. }Q\in\partial\Omega,
\end{equation}
where the a.e. convergence is taken with respect to the ${\mathcal H}^{n-1}$ Hausdorff measure on $\partial\Omega$.

\medskip

In \cite{DPcplx}, it was shown that a Moser iteration scheme could be applied in the presence of $p$-ellipticity to yield higher regularity of 
solutions. Precisely, the following two lemmas were proven.

\begin{lemma}\label{LpAvebig}
Let the matrix $A$ be $p$-elliptic for $p \geq 2$ and let $B$ have coefficients satisfying $B_i(x) \leq K\delta(x)^{-1}$.
Suppose that $u$ is a $W^{1,2}_{loc}(\Omega;\BBC)$ solution to $\mathcal L$ in $\Omega$. Then, for any ball $B_r(x)$ with $r < \delta(x)/4$,
\begin{equation}\label{psq}
\int_{B_r(x)} |\nabla u(y)|^2 |u(y)|^{p-2} dy \lesssim r^{-2} \int_{B_{2r}(x))} |u(y)|^p dy
\end{equation}
and
\begin{equation}\label{pAve-old}
\left(\dint_{B_{r}(x))} |u(y)|^q dy\right)^{1/q}
\lesssim  \left(\dint_{B_{2r}(x)} |u(y)|^2 dy\right)^{1/2}
\end{equation}
for all $q \in (2,\frac{np}{n-2}]$ when $n > 2$, and
where the implied constants depend only $p$-ellipticity and $K$.
When $n=2$, $q$ can be any number in $(2,\infty)$.
In particular, $|u|^{(p-2)/2} u$ belongs to $W^{1,2}_{loc}(\Omega;\BBC).$
\end{lemma}

\begin{lemma}\label{LpAve}
Let the matrix $A$ be $p$-elliptic for $p < 2$ and let $B$ have coefficients satisfying $B_i(x) \leq K \delta(x)^{-1}$.
Suppose that $u$ is a $W^{1,2}_{loc}(\Omega;\BBC)$ solution to $\mathcal L$ in $\Omega$. Then, for any ball $B_r(x)$ with $r < \delta(x)/4$ and any $\varepsilon>0$
\begin{equation}\label{psq2}
r^2\dint_{B_r(x)} |\nabla u(y)|^2 |u(y)|^{p-2} dy \le C_\varepsilon \dint_{B_{2r}(x)} |u(y)|^p dy+\varepsilon\left( \dint_{B_{2r}(x)} |u(y)|^2 dy\right)^{p/2}
\end{equation}
and
\begin{equation}\label{pAve2}
\left(\dint_{B_{r}(x)} |u(y)|^2 dy\right)^{1/2}
\le  C_\varepsilon\left(\dint_{B_{2r}(x)} |u(y)|^p dy\right)^{1/p}+\varepsilon \left(\dint_{B_{2r}(x)} |u(y)|^2 dy\right)^{1/2}
\end{equation}
where the constants depend only $p$-ellipticity and $K$.
In particular, $|u|^{(p-2)/2} u$ belongs to $W^{1,2}_{loc}(\Omega;\BBC).$
\end{lemma}



In \cite{FMZ}, two improvements were observed. First, the reverse H\"older inequality for $p<2$ was simplified, eliminating the term
containing the
integral of $|u|^2$ multiplied by $\varepsilon$ on the left hand side of  \eqref{pAve2}. Second, the method of proof led to an
extension of the reverse H\"older inequalities to the boundary, namely
for balls $B$ for which the $\mbox{Tr}(u) = 0$ on $2B \cap \partial \Omega$. The statement of the boundary reverse H\"older is as follows.

\begin{lemma}(\cite{FMZ}) \label{BRH-old}
Let $\Omega$ be a chord-arc domain. Let $\mathcal L = \partial_{i}\left(A_{ij}(x)\partial_{j}u\right) $ be a $q$-elliptic operator.
Let $u\in W$ be a weak solution to $\mathcal Lu = 0$ in $\Omega$ and  $B$ be a ball of radius $r$ centered on $\partial \Omega$ such that $\Tr u = 0$ on $2B \cap \partial \Omega$. There holds
\[\int_{B \cap \Omega} |u|^{q-2} |\nabla u|^2 \, dx  \leq \frac{C}{r^2} \int_{(2B \setminus B)\cap \Omega} |u|^q \, dx.\]
Furthermore, if $q>2$, one has
\[\left( \frac{1}{|B \cap \Omega|} \int_{B\cap \Omega}  |u|^q \, dx \right)^{\frac{1}{q}} \leq C  \left(\frac{1}{|2B \cap \Omega|} \int_{2B \cap \Omega}  |u|^2   \, dx \right)^{\frac{1}{2}},\]
and if $q<2$, we have
\[\left(\frac{1}{|B \cap \Omega|} \int_{B\cap \Omega}  |u|^2 \, dx \right)^{\frac{1}{2}} \leq C  \left(\frac{1}{|2B \cap \Omega|} \int_{2B \cap \Omega}  |u|^q   \, dx\right)^{\frac{1}{q}}.\]
The constant $C>0$ depends only on $n$, $q$, the constant $\lambda_q$ and $\|A\|_\infty$.
\end{lemma}

We claim the following improvement of this lemma holds.

\begin{lemma} \label{BRH}
Let $\Omega$ be a chord-arc domain. Let $\mathcal L = \partial_{i}\left(A_{ij}(x)\partial_{j}u\right) +B_{i}(x)\partial_{i}$ be a $q$-elliptic operator. There exists $K=K(n,q,\lambda_q,\|A\|_\infty)>0$ of the following significance. Suppose that $|B_i(x)| \le K\delta(x)^{-1}$ for all $x\in\Omega$. Let $u\in W$ be a weak solution to $\mathcal Lu = 0$ in $\Omega$ and  $B$ be a ball of radius $r$ centered on $\partial \Omega$ such that $\Tr u = 0$ on $2B \cap \partial \Omega$. There holds
\[\int_{B \cap \Omega} |u|^{q-2} |\nabla u|^2 \, dx  \leq \frac{C}{r^2} \int_{(2B \setminus B)\cap \Omega} |u|^q \, dx.\]
Furthermore, one has
\begin{equation}\label{eq-rh}
\left( \frac{1}{|B \cap \Omega|} \int_{B\cap \Omega}  |u|^q \, dx \right)^{\frac{1}{q}} \leq C  \left(\frac{1}{|2B \cap \Omega|} \int_{2B \cap \Omega}  |u|^p   \, dx \right)^{\frac{1}{p}}
\end{equation}
for any $p> 0$. The constant $C>0$ depends only on $n$, $q$, the constant $\lambda_q$ and $\|A\|_\infty$.
\end{lemma}

The first improvement is rather trivial, namely that the proof given in \cite{FMZ}  also holds for operators with lower order terms of the form $\mathcal L = \partial_{i}\left(A_{ij}(x)\partial_{j}\right) 
+B_{i}(x)\partial_{i}$ where $|B_i(x)| \le K\delta(x)^{-1}$ and $K$ is sufficiently small. This can be seen by examining the proof given in the paper.

The second improvement is that in the reverse H\"older inequality we can have any exponent $p>0$ on the right-hand side of the inequality. This observation is originally due to Fefferman-Stein \cite{FS} but we would like to refer the reader to a more recent exposition by Shen (\cite{S2}, Theorem 2.4) for more details than we provide here. The proof starts with knowledge that the reverse H\"older inequality holds for a specific pair of exponents $q>p$: in our situation when $q>2$, it holds for $p=2$. Then, an 
argument that employs the known reverse H\"older inequality on rescaled balls multiple times, ultimately yields
\begin{equation}
\left( \frac{1}{|B_{sr} \cap \Omega|} \int_{B_{sr}\cap \Omega}  |u|^q \, dx \right)^{\frac{1}{q}} \leq C s^{\frac{n}q}t^{\frac{n}2}(t-s)^{n(\frac1q-\frac12)} \left(\frac{1}{|B_{tr} \cap \Omega|} \int_{B_{tr} \cap \Omega}  |u|^2   \, dx \right)^{\frac{1}{2}},
\end{equation}
for any pair of boundary balls $B_{sr}$, $B_{tr}$ with same center and radii $sr$, $tr$ respectively for any $0<s<t<1$. Next, for a fixed $0<p<2$ we write $\frac12=\frac{1-\theta}{q}+\frac\theta{p}$ for $\theta\in(0,1)$. Using H\"older's inequality we have
\begin{eqnarray}
&& \left(\frac{1}{|B_{tr} \cap \Omega|} \int_{B_{tr} \cap \Omega}  |u|^2   \, dx \right)^{\frac{1}{2}}
\le\nonumber\\&&\left(\frac{1}{|B_{tr} \cap \Omega|} \int_{B_{tr} \cap \Omega}  |u|^q   \, dx \right)^{\frac{1-\theta}{q}}\left(\frac{1}{|B_{tr} \cap \Omega|} \int_{B_{tr} \cap \Omega}  |u|^p   \, dx \right)^{\frac{\theta}{p}}
 \end{eqnarray}
Finally, for $0<t<1$ let
$$I(t)=\left( \frac{1}{|B_{tr} \cap \Omega|} \int_{B_{tr}\cap \Omega}  |u|^q \, dx \right)^{\frac{1}{q}}\Bigg/\left(\frac{1}{|B_{r} \cap \Omega|} \int_{B_{r} \cap \Omega}  |u|^p   \, dx \right)^{\frac{1}{p}}.$$
Combining the two previous inequalities yields:
$$I(s)\le C s^{\frac{n}q}t^{\frac{n}2}(t-s)^{n(\frac1q-\frac12)} I(t)^{1-\theta}.$$
Next, we choose $s=t^b$ for some $b>1$ such that $b^{-1}>1-\theta$. We take log of the inequality above and then integrate it in $t$ with respect to $t^{-1}dt$ over the interval $[1/2,1]$. This finally yields
$$\left(\frac1b-\theta\right)\int_{1/2}^1\frac{\log I(t)}{t}dt\le C.$$
As $I(t)\ge cI(1/2)$ for all $t\in [1/2,1]$ we obtain $I(1/2)\le C$ which gives \eqref{eq-rh} when $p<2$. The case $p>2$ is easier as we can use the version of \eqref{eq-rh} when $p=2$ and then the usual H\"older inequality.

\begin{remark} Observe that we can apply same argument to the inequality \eqref{pAve-old}. Hence we have the following. For $u$ as in Lemma \ref{LpAvebig} we have
\begin{equation}\label{pAve}
\left(\dint_{B_{r}(x))} |u(y)|^q dy\right)^{1/q}
\lesssim  \left(\dint_{B_{2r}(x)} |u(y)|^p dy\right)^{1/p}
\end{equation}
for all $p>0$ and $q<p_0$ where $p_0=\sup\{p>1:\mbox{ matrix $A$ is $p$-elliptic}\}$.

\end{remark}

\section{Proof of Theorem \ref{MainThm}}

The proof is based on the following abstract result \cite{Sh1}, see also \cite[Theorem 3.1]{WZ} for a version on an arbitrary bounded domains. In both of these papers, the argument is
carried for the case $q=2$ below, but can be generalized as follows.

\begin{theorem}\label{th-sh} Let $\Omega$ be an open set in $\mathbb R^n$ and let $T$ be a bounded sublinear operator on $L^q(\partial \Omega;{\mathbb C}^m)$, $q>1$. Suppose that
for some $p>q$, $T$ satisfies the following $L^p$ localization property. For any ball $\Delta=\Delta_d\subset \partial \Omega$ and $C^\infty$ function $f$ supported in $\partial
\Omega \setminus 3\Delta$ the following estimate holds:
\begin{align}
&\left(|\Delta|^{-1}\int_\Delta|Tf|^p\,dx'\right)^{1/p}\le\label{eq-pf8}\\
&\qquad C\left\{\left(|2\Delta|^{-1}\int_{2\Delta}|Tf|^q\,dx'\right)^{1/q}+\sup_{\Delta'\supset \Delta}\left(|\Delta'|^{-1}\int_{\Delta'}|f|^q\,dx'\right)^{1/q}  \right\},\nonumber
\end{align}
for some $C>0$ independent of $f$. Then $T$ is bounded on $L^r(\partial \Omega;{\mathbb C}^m)$ for any $q\le r<p$.
\end{theorem}
In our case the role of $T$ is played by the sublinear operator $f\mapsto \tilde{N}_{2,a}(u)$, where $u$ is the solution of the Dirichlet problem ${\mathcal L}u=0$ with boundary data $f$. In the statement of the theorem above,
the specific enlargement factors ($2\Delta$, $3\Delta$) do not play a significant role. Hence it will suffice to establish estimate \eqref{eq-pf8}
with $2\Delta$ replaced by $8m\Delta$, and with $f$ vanishing on $16m\Delta$ for some $m>1$ to be determined later.

The operator $T:f\mapsto \tilde{N}_{2,a}(u)$ is sublinear and bounded on $L^q$, by assumption.
 To prove \eqref{eq-pf8} for this choice of $T$ we shall establish the following reverse H\"older inequality.
 
\begin{proposition}\label{P-RH} Let $a>0$ and let $\mathcal L$ and $p_0$ be as in Theorem \ref{MainThm}. Then for any $p\in (1,\frac{p_0(n-1)}{n-2})$ there exist $K(p)>0$ such that if $|B_i(x)|\le K(p)\delta(x)^{-1}$ holds for the first order coefficients of the operator $\mathcal L$ then
\begin{equation}
\left(\frac1{|\Delta|}\int_\Delta|\tilde{N}_{2,a}(u)|^p\,dx'\right)^{1/p}\le\label{eq-pf9}
C\left(\frac1{|8m\Delta|}\int_{8m\Delta}|\tilde{N}_{2,a}(u)|^q\,dx'\right)^{1/q},
\end{equation}
holds for all $1\le q\le p$. Here $u$ solves $\mathcal Lu=0$ in $\Omega$ and  $f=u\big|_{\partial\Omega}$ vanishes on $16m\Delta$.
\end{proposition}

 We are also free to choose the aperture $a$ for which we establish \eqref{eq-pf9}.
The norms of the nontangential maximal function operators with varying apertures are
all equivalent, up to a constant that depends on this aperature. (See \cite{MPT} or \cite{HMT} for a proof of this on CAD domains.)
With \eqref{eq-pf9} in hand, Theorem \ref{th-sh} gives
\begin{equation}\label{eq-pf10}
\|\tilde{N}_{2,a}(u)\|_{L^{r}({\BBR}^{n-1})}\le C\|f\|_{L^{r}({\BBR}^{n-1})},
\end{equation}
establishing $L^r$ solvability of the Dirichlet problem for the operator $\mathcal L$ for $q \leq r < p$, thus proving Theorem \ref{MainThm}.

\medskip

It remains to establish Proposition \ref{P-RH}. Let us define
\begin{align}\label{eq-pf11}
&{\mathcal M}_1(u)(Q)=\sup_{y\in\Gamma_a(Q)}\{u_2(y):\,\delta(y)\le d\},\\
&{\mathcal M}_2(u)(Q)=\sup_{y\in\Gamma_a(Q)}\{u_2(y):\,\delta(y)> d\}.\nonumber
\end{align}
Here $d=\mbox{diam}(\Delta)$ and $u_2$ is the $L^2$ average of $u$
$$u_2(y)=\left(\dint_{B_{\delta(y)/2}(y)}|u(z)|^2\,dz\right)^{1/2}.$$
It follows that
$$\tilde{N}_{2,a}(u)=\max\{{\mathcal M}_1(u),{\mathcal M}_2(u)\}.$$
We first estimate ${\mathcal M}_2(u)$. Pick any $Q \in\Delta$, and to this end we prove the following proposition,
which requires our modified corkscrew regions.

\begin{proposition}\label{Prop}
Let $\Delta$ be a boundary ball of radius $d$ and let $Q \in \Delta$.
Then for any $y\in\Gamma(Q)$ with $\delta(y)>d$, the set
$A := \{P \in 2\Delta: y \in \Gamma_a(P)\}$
has size comparable to $2\Delta$. 
\end{proposition}

\begin{proof}
Since $y\in\Gamma(Q)$, there exists a $Q_0 \in S_a(y)$ such that $Q \in \Delta(Q_0, a\delta(y))$. 
From the fact that $\delta(y) > d$ and using the Ahlfors-David regularity, we have that
$\sigma(\Delta(Q, d) \cap \Delta(Q_0,a\delta(y))) > cd^{n-1}$, for a constant $c$ depending
only on $a$ and on the chord-arc geometry. Moreover,
$\Delta(Q, d) \cap \Delta(Q_0,a\delta(y)) \subset A$, which proves the proposition.

\end{proof}

Moreover,
$$P \in A\quad\Longrightarrow\quad y\in\Gamma_a(P)\quad\Longrightarrow\quad u_2(y)\le \tilde{N}_{2,a}(u)(P).$$
Hence for any $Q \in\Delta$,
$${\mathcal M}_2(u)(Q)\le C\left(\frac1{|2\Delta|}\int_{2\Delta} \left[\tilde{N}_{2,a}(u)(P)\right]^q\,d\sigma(P)'\right)^{1/q}.$$
It remains to estimate ${\mathcal M}_1(u)$ on $\Delta$. Consider any $s \in (p'_0,p_0)$.
Recall the expressions of the form $|u|^{s/2-1}u$ that arise in Lemmas \ref{LpAvebig} and 
\ref{LpAve}, and set $v = |u|^{s/2-1}u$.
As in \eqref{w}, $v_2$ denotes the $L^2$ average over the appropriate interior ball.

We next claim that, given the fact that $u$ vanishes on $3\Delta\subset m\Delta$, we have for any $Q \in \Delta$ and 
for $x \in \Gamma_a(Q)$ with $\delta(x) = h$, and for $P \in C(Q,h):=\{P: x \in \Gamma_a(P), \,\,h/2 < |P -Q| < h \}$,

\begin{equation}\label{Conebound}
v_2(x)  \lesssim (hd)^{1/2} A_{\tilde{a}}(\nabla v)(P)
\end{equation}
where 
$$A_{\tilde{a}}^2(\nabla v)(P) = d^{-1} \int_{\Gamma_{\tilde{a}}^{2d}(P)} |\nabla v|^2(z) \delta(z)^{1-n} dz.$$
Here, the parameter $\tilde{a}$ will be determined later, and the truncated corkscrew region is defined to be 
$\Gamma_{\tilde{a}}^{2d}(P):= \Gamma_{\tilde{a}}(P) \cap B(P,2d)$.

\medskip

Proof of \eqref{Conebound}.  For any $P \in C(Q,h)$, since $x \in \Gamma_a(Q)$, it follows that $x \in \Gamma_{1+2a}(P)$ and so there is a sequence of corkscrew points $x_j$ associated to
 the point $P$ at scales 
$r_j \approx 2^{-j}h$, $j=0,1,2,\dots$ with $x_0=x$. By the Harnack chain condition, for each $j$ there is a number $N$ and a constant $C$ such that
there exists  $n \leq N$ balls $B^{(j)}_k$ of radius $\approx 2^{-j}h$ with $C B^{(j)}_k \subset \Omega$, $x_{j-1} \in B^{(j)}_1$, 
$x_j \in B^{(j)}_n$, and $B^{(j)}_k \cap  B^{(j)}_{k+1} \ne \emptyset$. 
Therefore we can find another chain of balls with the same properties for a larger but fixed choice of $N$ so that
$4 B^{(j)}_k \subset \Omega$ 
and $B^{(j)}_{k+1} \subset 2B^{(j)}_{k}$.

Considering the whole collection of balls $B_k^{(j)}$ for all $j=0,1,2,\dots$ and $k=1,2,\dots,n(j)\le N$ it follows that we have an infinite chain of balls, the first of which contains $x_0$, converging to the boundary point $P$, with the property that any pair of consecutive balls in
the chain have roughly the same radius and whose 4-fold enlargements are contained in $\Omega$. We relabel these
balls $B_j(x_j,r_j)$ with centers $x_j$ and radii $r_j \approx t^{-j}h$ for some $t < 1$ depending on $N$. 

We next claim that, for any $\varepsilon > 0$, 

\begin{equation}\label{Averages}
\left|\dint_{B_j} |v|^2(z) dz - \dint_{B_{j+1}} |v|^2(z) dz\right| \leq \varepsilon t^{-j} \dint_{B_j} |v|^2(z) dz + C_{\varepsilon, t} h \int_{2B_{j}} |\nabla v|^2(z) \delta(z)^{1-n} dz
\end{equation}

The argument proceeds in two steps.
The first step is to obtain \eqref{Averages} but with $2B_j$ on the left hand side. That is,

\begin{equation}\label{Averages2}
\left|\dint_{B_j} |v|^2(z) dz - \dint_{B_{j+1}} |v|^2(z) dz\right| \leq \varepsilon t^{-j} \dint_{2B_j} |v|^2(z) dz + C_{\varepsilon, t} h \int_{2B_{j}} |\nabla v|^2(z) \delta(z)^{1-n} dz
\end{equation}

In the second step, we show that
\begin{equation}\label{shrink}
\dint_{2B_j} |v|^2(z) dz \leq  \dint_{B_j} |v|^2(z) dz + \varepsilon t^{-j} \dint_{2B_j} |v|^2(z) dz +  
C_{\varepsilon, t} h \int_{2B_{j}} |\nabla v|^2(z) \delta(z)^{1-n} dz
\end{equation}

From \eqref{shrink}, we choose $\varepsilon>0$ small enough to see that 
\begin{equation}\label{shrink2}
\dint_{2B_j} |v|^2(z) dz \leq  2\dint_{B_j} |v|^2(z) dz +  C_{\varepsilon, t} h \int_{2B_{j}} |\nabla v|^2(z) \delta(z)^{1-n} dz
\end{equation}
and use this estimate in \eqref{Averages2} to obtain \eqref{Averages}.
Here, and in the following estimates, the constant $C_{\varepsilon,t}$ is not necessarily the same at each occurrence.

The arguments for \eqref{Averages2} and \eqref{shrink} are essentially the same - both are
essentially Poincar\'e-type inequalities with an application of Cauchy-Schwarz. We give the argument assuming that $v$ is differentiable, which
can be justified by replacing $v$ by a smooth approximation in the Sobolev space.

To prove \eqref{Averages2}, define the map $T(x) =  r_j/r_{j+1}(x-x_j) + x_{j+1}$ from $B_j$ to $B_{j+1}$. (In the case of \eqref{shrink}, $T$ is just dilation.)
Then
\begin{equation}
|v^2(T(x)) - v^2(x)| \leq \int_{\ell \in [x, T(x)]} |\nabla v^2(\ell)|\,\, d\ell
\end{equation}
where $ [x, T(x)]$ is the line from $x$ to $T(x)$.

Averaging $x$ over $B_j$, using the triangle inequality, and observing that the collection of lines $[x,T(x)]$ is contained 
in $2B_j$, gives 
\begin{equation}\label{FTC}
 \left|\dint_{B_j} |v|^2(z) dz - \dint_{B_{j+1}} |v|^2(z) dz\right| \leq  C't^{-j} h \dint_{2B_{j}} |v(z)| |\nabla v(z)| dz
\end{equation}

Applying Cauchy-Schwarz to the right hand side of \eqref{FTC} gives \eqref{Averages2}, noting that 
$\delta(z) \approx t^{-j} h$.

The claim \eqref{Conebound} results from summing the averages in \eqref{Averages} as follows.

Set $$U_j := \dint_{B_j} |v|^2(z) dz - \dint_{B_{j+1}} |v|^2(z) dz$$
Because $u$ vanishes on the boundary, the averages $\dint_{B_j} |v|^2(z) dz $ are converging to zero (\cite{DPcplx}, section 7). Therefore, for 
any choice of $\eta >0$, we choose $M$ large enough so that
$\dint_{B_{M}} |v|^2(z) dz < \eta$, and for $j<M$,
$$\dint_{B_j} |v|^2(z) dz = \sum_{k=j}^M U_k + \dint_{B_M} |v|^2(z) dz < \sum_{k=j}^M |U_k| + \eta.$$

From \eqref{Averages} together with the fact that the collection $2B_j$ has finite overlap and the union belongs to $\Gamma^d_a(y')$, 
\begin{equation}\label{Dsum}
\dint_{B_0} |v|^2(z) dz \leq \sum_{j=0}^M |U_j| + \eta \leq \eta + \varepsilon \sum_{j=0}^M t^{-j} (\sum_{k=j}^M |U_k| + \eta) +
C_{\varepsilon,t} h \int_{\Gamma_{\tilde{a}}^{2d}(P)} |\nabla v|^2 \delta(z)^{1-n} dz
\end{equation}
where the aperture $\tilde{a}$ is chosen sufficiently large (depending only on the constants defining the geometry of the domain $\Omega$) such that for each $j\ge 0$ we have
$2B_j \subset B(x_j, \delta(x_j)/2) \subset \Gamma_{\tilde{a}}^{2d}(P)$.

Interchanging the order of summation,
$
 \sum_{j=0}^M t^{-j} \sum_{k=j}^M |U_k| \leq \sum_{k=0}^M |U_k| \sum_{j \leq k} t^{-j} 
 $
makes it apparent that if we now choose $\varepsilon$ so that $\varepsilon < (1-t)/2$
in \eqref{Dsum}, then we have 
 \begin{equation}\label{sumbound}
 \dint_{B_0} |v|^2(z) dz \leq \sum_{j=0}^M |U_j| + \eta \leq 
C_{\varepsilon,t} h \int_{\Gamma_{\tilde{a}}^{2d}(P)} |\nabla v|^2 \delta(z)^{1-n} dz + 2\eta. 
\end{equation}
.

Letting $\eta \to 0$ gives a variant of \eqref{Conebound}, for the average $\dint_{B_0} |v|^2(z) dz$ rather 
than for $v^2_2(x) = \dint_{B(x, \delta(x)/2)} |v|^2(z) dz$. However, the estimate for $\dint_{B_0} |v|^2(z) dz$ is sufficient by
an argument exactly like that for \eqref{shrink2}.

\medskip

We use \eqref{Conebound} to average over $P \in C(Q,h)$. (We now omit reference to the aperture.)
Since $\sigma(C(Q,h))\approx h^{n-1}$, we have for $x \in \Gamma(Q) $ :

\begin{equation}\label{Abound}
v_2(x)  \lesssim h^{3/2-n}d^{1/2} \int_{C(Q,h)} A(\nabla v)(P) d\sigma(P)
\lesssim  d^{1/2} \int_{C(Q,h)} \frac{A(\nabla v)(P)}{|P - Q|^{n-3/2}} d\sigma(P)
\end{equation}

Because $C(Q, h) \subset 2\Delta$, we see that
\begin{equation}\label{M1bound}
\mathcal M_1(v_2)(Q) \lesssim d^{1/2} \int_{2\Delta} \frac{A(\nabla v)(P)}{|P - Q|^{n-3/2}} d\sigma(P)
\end{equation}

By the fractional integral estimate, this implies that
\begin{eqnarray}\label{eq-pf13}
\left(\frac1{|\Delta|}\int_\Delta[{\mathcal M}_1(v_2)(P)]^r\,d\sigma(P)\right)^{1/r}\le & \nonumber \,Cd
\left(\frac1{|2\Delta|}\int_{2\Delta}[A(\nabla v)(P)]^2\,d\sigma(P)\right)^{1/2}\\
\le &Cd \left(\frac1{|T(m\Delta)|}\int_{T(m\Delta )} |\nabla v(x)|^2\,dx\right)^{1/2}
\end{eqnarray}
where $\frac1r=\frac12-\frac1{2(n-1)}$ and $m=m(\tilde{a})>2$ is such that $T(m\Delta)$ contains all points $x\in \Gamma_{\tilde{a}}^{2d}(P)$ 
for any $P\in 2\Delta$.

To further estimate \eqref{eq-pf13} we use the Lemma \ref{BRH}, recalling that $|\nabla v|^2 = |u|^{s-2}|\nabla u|^2$:
$$\left(\frac1{|T(m\Delta)|}\int_{T(m\Delta)} |\nabla v(x)|^2\,dx\right)^{1/2}
\lesssim d^{-1}\left(\frac1{|T(2m\Delta)|}\int_{T(2m\Delta)} |u(x)|^s\,dx\right)^{1/2},$$
whenever the solution $\mathcal Lu=0$ vanishes on at least $3m \Delta$.

By Lemma \ref{BRH}, we therefore have that
\begin{eqnarray}\label{eq-pf13a}
\left(\frac1{|\Delta|}\int_\Delta[{\mathcal M}_1(v_2)(P)]^r\,d\sigma(P)\right)^{1/r}\le & \nonumber \,C
\left(\frac1{|T(2m\Delta)|}\int_{T(2m\Delta)} |u(x)|^s\,dx\right)^{1/2}\\
\le &C \left(\frac1{|T(4m\Delta)|}\int_{T(4m\Delta)} |u(x)|^q\,dx\right)^{s/2q}
\end{eqnarray}

Rewriting \eqref{eq-pf13a} in terms of $u$, and choosing $rs/2 = p$, gives

\begin{eqnarray}\nonumber
\left(\frac1{|\Delta|}\int_\Delta[{\mathcal M}_1(u_s)(P)]^p\,d\sigma(P)\right)^{1/p}\le &
C \left(\frac1{|T(4m\Delta)|}\int_{T(4m\Delta)} |u(x)|^q\,dx\right)^{1/q} \\
\le &C \left(\frac1{|4m\Delta|}\int_{4m\Delta} [\tilde N_q(u)(P]^q \,d\sigma(P)\right)^{1/q}.\nonumber 
\end{eqnarray}

The final step is to replace the $u_s$ averages by $u_2$ ones, as well as the $\tilde{N}_q$ by $\tilde{N}_2$. This can be done thanks to \eqref{pAve} (see \cite[Proposition 3.5]{DPcplx} or the corresponding statement in \cite{FMZ}) to give us 
\begin{equation}
\left(\frac1{|\Delta|}\int_\Delta[{\mathcal M}_1(u_2)(P)]^p\,d\sigma(P)\right)^{1/p}\le
C \left(\frac1{|8m\Delta|}\int_{8m\Delta} [\tilde N_2(u)(Q)]^q \,d\sigma(Q)\right)^{1/q} 
\end{equation}

We now conclude that \eqref{eq-pf9}
holds for $p=rs/2$. Since $r=2(n-1)/(n-2)$ this implies that $p=s(n-1)/(n-2)$. Given that we can take any $s\in (p_0',p_0)$
the claim of Proposition \ref{P-RH} follows.
\mbox

\end{document}